\documentclass{amsart}
\usepackage{amssymb,latexsym,graphics,enumerate,color}
\usepackage{fullpage}
\usepackage{graphicx}
\usepackage[font=small,labelfont=bf]{caption}
\numberwithin{equation}{section}
\usepackage{bbm}
\usepackage[normalem]{ulem}

\usepackage{xcolor,cancel}

\newtheorem{theorem}{Theorem}[section]
\newtheorem{corollary}[theorem]{Corollary}
\newtheorem{lemma}[theorem]{Lemma}
\newtheorem{proposition}[theorem]{Proposition}

\newcommand{\la}{\langle}
\newcommand{\ra}{\rangle}

\renewcommand{\O}{{\mathcal{O}}}

\newcommand{\R}{{\mathbb{R}}}

\newcommand{\cd}{{\,\cdot\,}}

\newcommand{\ang}{{\not\negmedspace\nabla}}
\newcommand{\good}{{\not\negmedspace\partial}}

\renewcommand{\S}{{\mathbb{S}}}

\newcommand{\tC}{{\tilde{C}}}
\newcommand{\ttC}{{\tilde{\tilde{C}}}}

\begin{document}
\bibliographystyle{plain}

\title
{
On a system of weakly null semilinear wave equations
}

\author{Jason Metcalfe}
\author{Alexander Stewart}
\address{Department of Mathematics, University of North Carolina, Chapel Hill}
\email{metcalfe@email.unc.edu, stew314@live.unc.edu}


\thanks{The first author was supported in part by Simons Foundation
  Collaboration Grant 711724, and both authors received support from
  NSF grant DMS-2054910 (PI Metcalfe).  The results contained herein
  were developed as a part of the Undergraduate Honors Thesis of AS}

 \begin{abstract}
We develop a new method for addressing certain weakly null systems of
wave equations.  This approach does not rely
on Lorentz invariance nor on the use of null foliations, both of which
restrict applications to, e.g., multiple speed systems.  The proof
uses a class of space-time Klainerman-Sobolev estimates of the first
author, Tataru, and Tohaneanu, which pair
nicely with local energy estimates that combine the $r^p$-weighted
method of Dafermos and Rodnianski with the ghost weight method of
Alinhac.  We further refine the standard local energy estimate with a
modification of the $\partial_t-\partial_r$ portion of the multiplier.
 \end{abstract}

\maketitle


\section{Introduction}
This article represents a proof of concept for a method of addressing
certain systems of weakly null wave equations that do not satisfy the classical
null condition.  This example falls into the class of equations
studied in \cite{Hidano-Yokoyama}.  For
simplicity of exposition, we only consider a semilinear system.  
Unlike
\cite{Hidano-Yokoyama}, the methods here do not use the Lorentz
boosts, which is important for similar problems in the setting of
multiple speeds, exterior domains, or stationary asymptotically flat
background geometry.  And when compared to the methods of
\cite{Keir}, which apply to a broader class of weakly null equations,
we believe that our methods are simpler and, as we do not rely on null
foliations, additional applications to multiple speeds systems appear
possible.  The current work is most akin to that of
\cite{Hidano-Zha}, which is based on the ideas of \cite{KlSid} and
proves global existence without the use of the Lorentz boosts, but we
believe our method to have added flexibility for other applications.

In three spatial dimensions, it is known that solutions to semilinear
systems of equations of the 
form $\Box u = Q(\partial u)$ with nonlinearity that vanishes to
second order at the origin can only be guaranteed to exist almost
globally, which means that the lifespan grows exponentially as the
size of the data shrinks.  See, e.g., \cite{John-Klainerman} for the lower bound
on the lifespan and \cite{John_counterexample}
and\cite{Sideris_counterexample} for counterexamples to global 
existence.  Based on the fact that the components of the space-time
gradient $\partial u = (\partial_t u, \nabla_x u)$ that are tangent to
the light cone decay faster, the null condition was identified in
\cite{Christodoulou} and \cite{Klainerman} as a sufficient condition for guaranteeing
  small data global existence.  This condition requires that at least
  one factor of each nonlinear term (at the quadratic level) to be one of
  the ``good'' directions.  Einstein's equations, for example, do not
  satisfy this classical null condition, which led to the introduction
  of the weak null condition in \cite{Lindblad-Rodnianski-1, Lindblad-Rodnianski} as a possible sufficient
  condition for small data global existence.  Further evidence
  supporting this is given in \cite{Keir}.

  Here we shall consider a coupled system of equations.  One of the
  equations satisfies the classical null condition, but the other does
  not.  The intuition is that the equation satisfying the null
  condition has a solution that decays faster, and when that is
  plugged into the second equation, this additional decay allows for
  an argument to be closed.

  We specifically will consider
\begin{equation}
  \label{main_equation}
  \begin{cases}
    \Box u = \partial_tu\partial_t v - \nabla u\cdot\nabla v,\\ 
    \Box v = \partial_t v \partial_t u,\\
     (u(0,\cd),\partial_tu(0,\cd))=(u_{(0)},u_{(1)}),\\
    (v(0,\cd),\partial_t v(0,\cd))=(v_{(0)},v_{(1)}).
  \end{cases}
\end{equation}
For simplicity of exposition, we shall take the initial data to be
compactly supported, say within $\{|x|\le 2\}$.

In order to describe the ``good'' directions, we shall frequently
orthogonally decompose the (spatial) gradient into radial and angular portions:
\[\nabla = \frac{x}{r}\partial_r + \ang.\]
The directions that are tangent to the light cone are
\[\good = (\partial_t+\partial_r, \ang).\]
By noting that
\[\partial_t u\partial_t v - \nabla u\cdot\nabla v =
(\partial_t + \partial_r )u\partial_t v -\partial_r u(\partial_t +
\partial_r)v - \ang u \cdot \ang v,
  \]
we see that the equation for $u$ satisfies the null condition.  The
equation for $v$, however, does not.  Nevertheless, we shall prove
that solutions to \eqref{main_equation} with sufficiently small
initial data exist globally.

Our main theorem is the following statement of global existence.
\begin{theorem}\label{main_theorem}
Suppose that $u_{(j)}, v_{(j)}\in C^\infty_c(\R^3)$.  Then there is a
$N\in \mathbb{N}$ sufficiently large and $\varepsilon_0>0$
sufficiently small so that if
  \begin{equation}
    \label{smallness}
    \sum_{|\alpha|\le N+1} \|\partial_x^{\alpha} u_{(0)}\|_{L^2}
+ \sum_{|\alpha|\le N+1} \|\partial_x^{\alpha} v_{(0)}\|_{L^2}
+ \sum_{|\alpha|\le N} \|\partial_x^{\alpha}u_{(1)}\|_{L^2}
+\sum_{|\alpha|\le N} \|\partial_x^\alpha v_{(1)}\|_{L^2}
\le \varepsilon
  \end{equation}
with $\varepsilon\le \varepsilon_0$, then \eqref{main_equation} has a
unique global solution $(u,v)\in (C^\infty([0,\infty)\times \R^3))^2$.
\end{theorem}

The methods that we employ are partly inspired by \cite{KSS} where
almost global existence was established for equations without a null
condition by pairing a local energy estimate with a weighted Sobolev
estimate that provides decay in $|x|$ rather the $t$.  
The latter does not require the use of any time dependent vector
fields, which was instrumental in adapting the method of invariant
vector fields to, e.g., exterior domains.  The paper \cite{dafermos_rodnianski}
developed the $r^p$-weighted local energy estimate.  In this variant
of the local energy estimate, the additional decay for the ``good''
derivatives manifests itself as much improved weights.  In
\cite{Facci-Mcentarrfer-M}, an analog of \cite{KSS} was established using these
$r^p$-weighted estimates in order to show global existence for wave
equations with the null condition.

In \cite{M-Rhoads}, the $r^p$-weighted multiplier of \cite{dafermos_rodnianski} was
combined with a ``ghost weight'' as in \cite{Alinhac_ghostweight}.
The resulting estimate allowed for additional improvements on the
weight of $(\partial_t+\partial_r)u$ near the light cone.  This was
then combined with the space-time Klainerman-Sobolev estimates of
\cite{MTT} in order to establish long-time existence for systems of
wave equations where the nonlinearity is allowed to depend on the
solution not just its derivative.  We rely strongly upon these ideas.
A further modification of the $(\partial_t-\partial_r)$ component
of the multiplier for typical local energy estimates is introduced here.
This modification, in particular, while requiring a faster decaying
weight also provides a more rapidly decaying weight on the forcing
term.  

\subsection{Notation}
Here we fix some notation that will be used throughout the paper.  We
let
\[\Omega=x\times \nabla,\quad S=t\partial_t+r\partial_r,\quad
  Z=(\partial_t, \nabla, \Omega, S)\] 
denote the admissible vector fields.  We will use the shorthand
\[|Z^{\le N} u|=\sum_{|\mu|\le N} |Z^{\mu} u|,\quad |\partial^{\le N}
  u|=\sum_{|\mu|\le N} |\partial^\mu u|.\]
A key property of the vector fields $Z$ is that they all preserve
solutions to the homogeneous wave equation since
\[[\Box, \partial]=[\Box,\Omega]=0,\quad [\Box, S] = 2\Box.\]
It will also be important to notice that
\begin{equation}
  \label{commutators}
  [Z,\partial]\in\text{span}(\partial),\quad |[Z, \good]u|\le
  \frac{|Zu|}{r} + |\good u|.
\end{equation}
In the proof of local energy estimates, we will frequently use that
\begin{equation}
  \label{angCommutator}
  [\nabla, \partial_r] = [\ang, \partial_r] = \frac{1}{r}\ang.
\end{equation}
We, moveover, note that
\begin{equation}
  \label{angBound}
  \ang = -\frac{x}{r^2}\times \Omega,\quad |\ang u|\le \frac{1}{r}|Zu|.
\end{equation}

We will often decompose $\R^3$ into (inhomogeneous) dyadic
regions.  To that end, let
\[A_{R}=\{R\le \la x\ra\le 2R\},\quad  
\tilde{A}_R = \Bigl\{\frac{7}{8}R\le \la x\ra \le
\frac{17}{8}R\Bigr\}.\]
Similarly, we set
\[X_U = \{(t,x)\in \R_+\times\R^3\,:\, U\le \la t-r\ra\le 2U\},\]
with $\tilde{X}_U$ denoting a similar enlargement.

We shall use a finer refinement, as in \cite{MTT}, when necessary.
Because of our assumption that the initial data are supported in
$\{|x|\le 2\}$ and because of the finite speed of propagation, it will
suffice to examine $C=\{r\le t+2\}$.  We then consider a dyadic strip
\[C_\tau = \{(t,x)\in \R_+\times \R^3\,:\, \tau\le t\le 2\tau, \: r\le t+2\}.\]

Away from the light cone $t=|x|$, we further decompose into dyadic
regions in the $r$ variable:
\[C^{R=1}_\tau = C_\tau \cap \{r\le 2\},\quad  C^R_\tau = C_\tau\cap
  \{R\le r\le 2R\} \text{ when } 1<R\le \tau/4.\]
We additionally set
\[\tilde{C}^{R=1}_\tau = C \cap \Bigl\{\frac{7}{8}\tau \le t\le \frac{17}{8}\tau, r\le
  \frac{17}{8}\Bigr\},\quad \tilde{C}^R_\tau = C \cap \Bigl\{\frac{7}{8}\tau\le t \le
  \frac{17}{8}\tau, \frac{7}{8}R\le r\le
  \frac{17}{8}R\Bigr\} \text{ when } 1<R\le \tau/4\]
to denote slight enlargements, which will accommodate the tails of the
cutoff functions that are used to localize in the sequel.  The key
property is that
\[\la r\ra\approx R,\quad t-r\approx \tau \quad \text{ on } C^R_\tau,
  \tilde{C}^R_\tau \quad \text{ with } \tau \ge 4 \text{ and } 1\le R\le \tau/4.\]

In the vicinity of the light cone, we instead dyadically decompose in
$t-|x|$.  To this end, let
\[C^{U=1}_\tau =
  C_\tau\cap \{|t-r|\le 2\},\quad C^U_\tau = C_\tau\cap
  \{U\le t-r\le 2U\} \text{ when } 1<R\le \tau/4.\]
As above, we denote a slight enlargement on both scales by
\[\tilde{C}^{U=1}_\tau = C \cap
  \Bigl\{\frac{7}{8}\tau\le t\le \frac{17}{8}\tau, |t-r|\le
  \frac{17}{8}\Bigr\},\]
and 
\[  \tilde{C}^U_\tau = C \cap
  \Bigl\{\frac{7}{8}\tau\le t\le \frac{17}{8}\tau, \frac{7}{8}U\le
  t-r\le \frac{17}{8}U\Bigr\} \text{ when } 1<U\le \tau/4.\]
These choices give
\[r \approx \tau,\quad \la t-r\ra \approx U \quad \text{ on } C^U_\tau,
  \tilde{C}^U_\tau \quad \text{ with } \tau \ge 4 \text{ and } 1\le U\le \tau/4.\]

With these notations in place, we have
\[C_\tau =  \Bigl(\bigcup_{1\le R\le
    \tau/4} C^R_\tau\Bigr)\cup\Bigl(\bigcup_{1\le U\le \tau/4}
  C^U_\tau\Bigr) \cup C^{\tau/2}_\tau\]
where
\[C^{\tau/2}_\tau = C_\tau \cap \{t-r\ge \tau/2\} \cap \{r\ge
  \tau/2\}.\]
On $C^{\tau/2}_\tau$, we have $r\approx \tau$ and $t-r\approx \tau$.
We may regard this region as either a $C^R_\tau$ or a $C^U_\tau$ region.


\begin{center}
    \includegraphics[width=.6\textwidth]{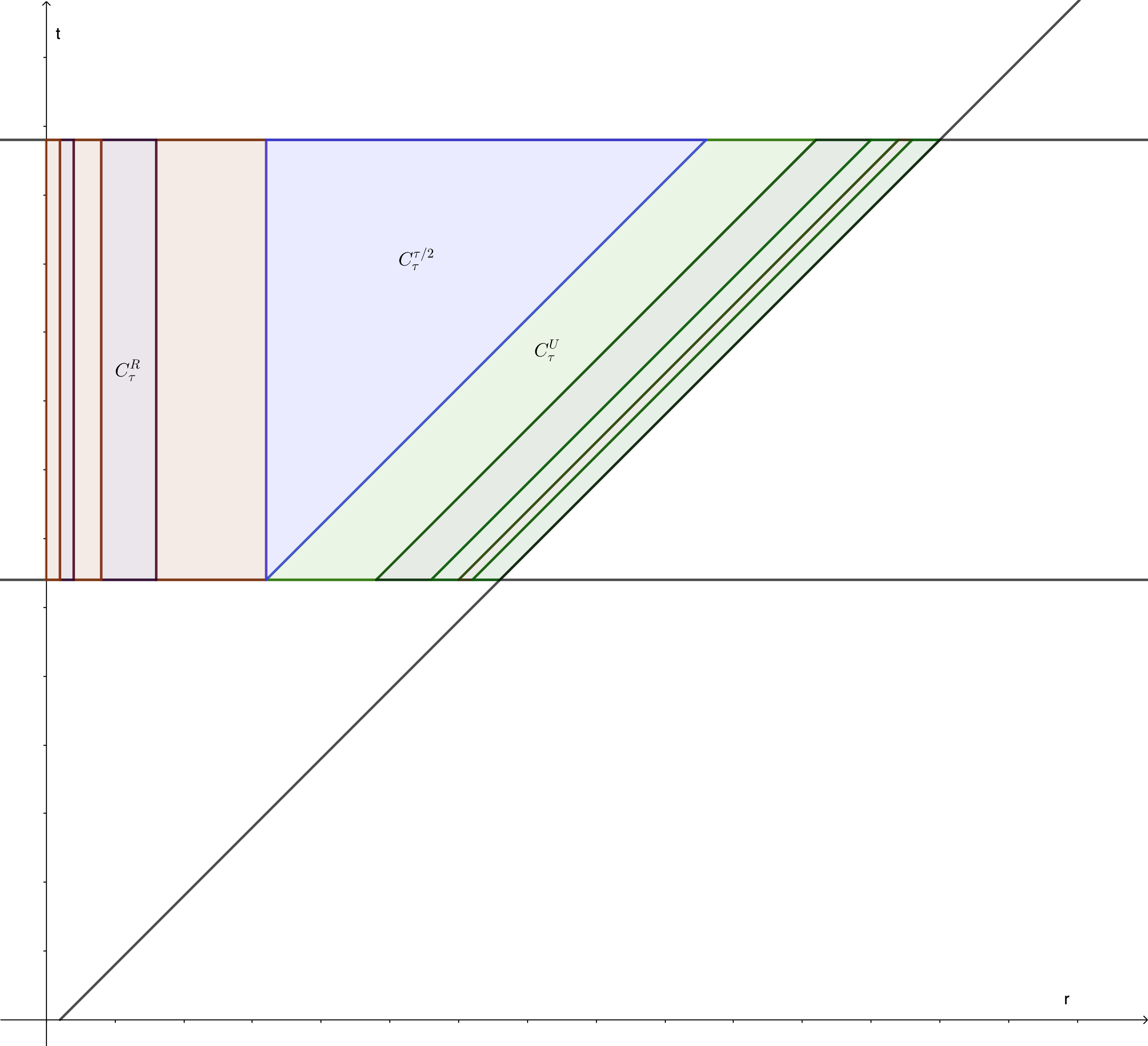}
  \captionof{figure}{The decomposition of $C_\tau$ into $C^R_\tau$ and $C^U_\tau$ regions}
\end{center}

On occasion, we shall use $\ttC^R_\tau, \ttC^U_\tau$ to denote
an enlargement of $\tC^R_\tau, \tC^U_\tau$ respectively.  In the sequel, it will be understood that $\tau, R, U$ always run
over dyadic values.

In order to localize to such regions, we fix the following notation
for cutoff functions.  Let $\chi$ be a smooth, nonnegative function
so that $\chi(z)\equiv 1$ for $z\ge 1$ and $\chi(z)\equiv 0$ for $z\le
7/8$.  We also set
\[\beta(z)=\chi(z)-\chi\Bigl(z-\frac{9}{8}\Bigr)\] so that
$\beta(z)\equiv 1$ for $1\le z\le 2$ and $\beta(z)\equiv 0$ when
$z\not\in [7/8, 17/8]$.

We will frequently use the mixed norm notation
\[\|u\|^p_{L^p L^q} = \int_0^\infty \|u(t,\cd)\|_{L^q(\R^3)}^p\,dt,\]
with the obvious alteration when $p=\infty$.  Unless specified, the
domain of all mixed norms of this type is $\R_+\times \R^3$.  We also fix the
following local energy norms, which will be discussed more in the next section:
\[\|u\|_{LE} = \sup_{R\ge 1} R^{-\frac{1}{2}}\|u\|_{L^2L^2(\R_+\times
    A_R)},\quad \|u\|_{LE^1}=\|(\partial u, u/r)\|_{LE}.\]


  \section{Local Energy Estimates}
  The integrated local energy estimate
  \begin{equation}\label{LE}
\|u\|_{LE^1}^2+
    \|\partial u\|^2_{L^\infty L^2}\lesssim
    \|\partial u(0,\cd)\|^2_{L^2} + \int_0^\infty \int |\Box u|\Bigl(|\partial
    u| + \frac{|u|}{r}\Bigr)\,dx\,dt
  \end{equation}
  is frequently proved by pairing the equation $\Box u$ with a
  multiplier of the form $\partial_t + f(r)\partial_r + \frac{f(r)}{r}
  u$, integrating over a space-time slab, and integrating by parts.
  The function $f(r)$ needs to be $C^2$, bounded, non-negative,
  increasing, and satisfy
  $-\Delta(f(r)/r) \ge 0$, which the function $f(r) = r/(r+R)$
  appropriately satisfies.  See \cite{Sterbenz}, \cite{MS, ms_mathz}.
  We may rewrite this multiplier as
  \begin{equation}\label{multiplier}\partial_t + \frac{r}{r+R}\partial_r + \frac{1}{r+R}=
  \frac{R/2}{r+R}\Bigl(\partial_t-\partial_r-\frac{1}{r}\Bigr) +
  \frac{r+(R/2)}{r+R} \Bigl(\partial_t+\partial_r+\frac{1}{r}\Bigr),\end{equation}
which has the property that the coefficient of
$\partial_t-\partial_r - \frac{1}{r}$ is nonnegative and decreasing in
$r$,
while the coefficient of $\partial_t+\partial_r+\frac{1}{r}$ is
nonnegative and increasing.  While there are other requirements, this
is the key
observation to allow for generalizations of the multiplier.  In
particular, we shall later consider
\[(1+r)^{-\delta} \Bigl(\partial_t - \partial_r - \frac{1}{r}\Bigr) +
  (1+r)^p e^{-\sigma_U(t-r)} \Bigl(\partial_t+\partial_r +
  \frac{1}{r}\Bigr).\]
Here $\sigma_U(z) = z/(U+|z|)$, $\delta > 0$, and $0<p<2$.

Multipliers of the form $r^p \Bigl(\partial_t + \partial_r +
\frac{1}{r}\Bigr)$ appeared previously in \cite{dafermos_rodnianski}
and $e^{-\sigma(t-r)}\partial_t$ in \cite{Alinhac_ghostweight}.  The
combination of the two as reflected above is originally from
\cite{M-Rhoads}.  The change in multiplier on $\partial_t -
\partial_r - \frac{1}{r}$ 
provides an additional degree of decay on the forcing term that helps to close the
nonlinear arguments in the sequel.

We first record a corollary of \eqref{LE}.
\begin{proposition}
Suppose $u\in C^2(\R_+\times \R^3)$ and for all $t\in \R_+$,
$|\partial^{\le 1} u(t,x)|\to 0$ as $|x|\to \infty$.  Then
  \begin{equation}
    \label{LE1}
    \|u\|_{LE^1} + \|\partial u\|_{L^\infty L^2} \lesssim \|\partial
    u(0,\cd)\|_{L^2} + \int_0^\infty \|\Box u(t,\cd)\|_{L^2}\,dt.
  \end{equation}
\end{proposition}

The proposition follows immediately from \eqref{LE} upon applying the
Schwarz inequality to see that
\[\int_0^\infty \int |\Box u|\Bigl(|\partial u| +
  \frac{|u|}{r}\Bigr)\,dx\,dt
  \le \Bigl(\|\partial u\|_{L^\infty L^2} + \|r^{-1} u\|_{L^\infty
    L^2}\Bigr) \int_0^\infty \|\Box u(t,\cd)\|_{L^2}\,dt.\]
A Hardy inequality gives
\[\|r^{-1} u\|_{L^2}\lesssim \|\partial_r u\|_{L^2},\]
which permits the first factor above to be bootstrapped.

We will now discuss the mixed $r^p$-weighted and ghost weighted
estimates of \cite{M-Rhoads}, where the former is motivated by
\cite{dafermos_rodnianski} and the latter by
\cite{Alinhac_ghostweight}.  To begin, we look at a variant of the
Hardy inequality that holds in the space-time norms and yields a
``good'' derivative.  This, in essence, previously appeared in \cite{M-Rhoads}.
\begin{lemma} Fix $0< p<2$.  Suppose $u\in C^2(\R_+\times\R^3)$ and
for every $t\in \R_+$, $r^{\frac{p}{2}}|u(t,x)|\to 0$ as $|x|\to
\infty$.  Then,
\begin{equation}
  \label{Hardy}
 \|\la r\ra^{\frac{p-1}{2}}r^{-1} u\|_{L^2L^2} + \|\la
 r\ra^{\frac{p-1}{2}}r^{-\frac{1}{2}}u\|_{L^\infty L^2}
\lesssim \|\la r\ra^{\frac{p-1}{2}}r^{-\frac{1}{2}}u(0,\cd)\|_{L^2} +
\|\la r\ra^{\frac{p-1}{2}} r^{-1}(\partial_t+\partial_r)(ru)\|_{L^2L^2}.
\end{equation}
\end{lemma}

\begin{proof}
We write
\[\int_0^T\int \frac{(1+r)^{p-1}}{r^2} u^2\,dx\,dt
= - \int_0^T \int_{\S^2}\int_0^\infty
(1+r)^{p-1}\Bigl[(\partial_t+\partial_r)(r^{-1})\Bigr]
(ru)^2\,dr\,d\omega\,dt.\]
Here $d\omega$ is the volume form on $\S^2$.
Integration by parts gives that this is
\begin{multline*}
  =-\int \frac{(1+r)^{p-1}}{r} u^2(T,x)\,dx + \int \frac{(1+r)^{p-1}}{r}u^2(0,x)\,dx
\\+(p-1)\int_0^T\int \frac{(1+r)^{p-2}}{r} u^2\,dx\,dt
+2\int_0^T\int \frac{(1+r)^{p-1}}{r^2} u (\partial_t+\partial_r)(ru)\,dx\,dt.
\end{multline*}
Using that $\frac{1}{1+r}\le \frac{1}{r}$ in the third term and 
applying the Schwarz inequality to the last term then shows that
\begin{multline*}
(1-|p-1|)  \int_0^T\int \frac{(1+r)^{p-1}}{r^2} u^2\,dx\,dt + \int
\frac{(1+r)^{p-1}}{r} u^2(T,x)\,dx 
\le \int \frac{(1+r)^{p-1}}{r}u^2(0,x)\,dx
\\+ 2 \Bigl(\int_0^T \int \frac{(1+r)^{p-1}}{r^2}
u^2\,dx\,dt\Bigr)^{\frac{1}{2}}\Bigl(\int_0^T\int
(1+r)^{p-1}\Bigl(r^{-1}(\partial_t+\partial_r)(ru)\Bigr)^2\,dx\,dt\Bigr)^{\frac{1}{2}}. 
\end{multline*}
Bootstrapping the first factor of the last term and taking a supremum
over $T$ then yields \eqref{Hardy}.
\end{proof}

We next record what, in essence, is the main new estimate of \cite{M-Rhoads}.  
\begin{proposition}
Fix $0<p<2$.  If $u\in C^2(\R_+\times\R^3)$ and
$r^{\frac{p+2}{2}}|\partial^{\le 1} u(t,x)|\to 0$ as $|x|\to \infty$, then
  \begin{multline}
    \label{MR}
  \|\la r\ra^p (\partial_t+\partial_r)u\|_{L^\infty L^2} +
  \|\la r\ra^p \ang u\|_{L^\infty L^2} + \|\la
  r\ra^{\frac{p-1}{2}}r^{-\frac{1}{2}} u\|_{L^\infty L^2}
\\+ \|\la r\ra^{\frac{p-1}{2}}
  (\partial_t + \partial _r)u\|_{L^2L^2} + \|\la
  r\ra^{\frac{p}{2}} r^{-\frac{1}{2}} \ang u\|_{L^2L^2}
+ \|\la r\ra^{\frac{p-1}{2}} r^{-1} u\|_{L^2L^2}
  \\+ \sup_{U\ge 1}U^{-\frac{1}{2}}\|\la
  r\ra^{\frac{p}{2}} 
  r^{-1}(\partial_t+\partial_r)(ru)\|_{L^2L^2(X_U)}
  \\ \lesssim \|\la r\ra^{\frac{p-1}{2}}r^{-\frac{1}{2}}u(0,\cd)\|_{L^2}+
  \|\la r\ra^{\frac{p}{2}} (\partial_t+\partial_r)u(0,\cd)\|_{L^2} +
  \|\la r\ra^{\frac{p}{2}} \ang u(0,\cd)\|_{L^2}
  \\+
  \Bigl(\sum_{\tau} \sum_{R\le \tau/4} \|\la r\ra^{\frac{p+1}{2}}\Box
  u\|^2_{L^2L^2(C^R_\tau)}\Bigr)^{\frac{1}{2}}
  + \sum_U \Bigl(\sum_{\tau \ge 4U} U 
  \|\la r\ra^{\frac{p}{2}}\Box u\|^2_{L^2L^2(C^U_\tau)}\Bigr)^{\frac{1}{2}}.
  \end{multline}
\end{proposition}

\begin{proof}
  Noting that
  \begin{equation}
  \label{conjugate}
  \Box u =
  r^{-1}\Bigl(\partial_t^2-\partial_r^2-\ang\cdot\ang\Bigr)(ru), \quad
  \Bigl(\partial_t+\partial_r + \frac{1}{r}\Bigr) u = r^{-1}\Bigl(\partial_t+\partial_r\Bigr)(ru),
\end{equation}
we  consider
\begin{multline*}\int_0^T\int (1+r)^p e^{-\sigma_U(t-r)}\Box u
  \Bigl(\partial_t+\partial_r+\frac{1}{r}\Bigr)u\,dx\,dt \\=
  \int_0^T\int_{\S^2}\int_0^\infty
  (1+r)^pe^{-\sigma_U(t-r)}\Bigl(\partial_t^2-\partial_r^2-\ang\cdot\ang\Bigr)(ru)
  \Bigl(\partial_t+\partial_r\Bigr)(ru)\,dr\,d\omega\,dt 
\end{multline*}
for $0<p<2$, which, using \eqref{angCommutator}, is equivalent to
\begin{multline*}\frac{1}{2}\int_0^T\int_{\S^2}\int_0^\infty (1+r)^p e^{-\sigma_U(t-r)}
  \Bigl(\partial_t-\partial_r\Bigr)\Bigl[\Bigl(\partial_t+\partial_r\Bigr)(ru)\Bigr]^2\,dr\,
  d\omega \,dt 
  \\+ \frac{1}{2}\int_0^T\int_{\S^2}\int_0^\infty (1+r)^p e^{-\sigma_U(t-r)}
  \Bigl(\partial_t+\partial_r\Bigr)|\ang (ru)|^2\,dr\,d\omega\,dt
  \\+ \int_0^T\int_{\S^2}\int_0^\infty \frac{(1+r)^p}{r}e^{-\sigma_U(t-r)}|\ang(ru)|^2\,dr\,d\omega\,dt.
\end{multline*}
Subsequent integrations by parts then give
\begin{multline}
  \label{plus}
  \int_0^T\int (1+r)^p e^{-\sigma_U(t-r)}\Box u
  \Bigl(\partial_t+\partial_r+\frac{1}{r}\Bigr)u\,dx\,dt
  \\=\frac{1}{2}\int_{\S^2}\int_0^\infty (1+r)^p e^{-\sigma_U(t-r)}
  \Bigl\{\Bigl[\Bigl(\partial_t+\partial_r\Bigr)(ru)\Bigr]^2+|\ang(ru)|^2\Bigr\}\,dr\,d\omega
  \Bigl|_{t=0}^T 
  \\
+ \frac{1}{2}\int_0^T \int_{\S^2} e^{-\sigma_U(t)} u^2(t,0)\,d\omega\,dt
  +  \frac{p}{2}\int_0^T\int_{\S^2}\int_0^\infty (1+r)^{p-1} e^{-\sigma_U(t-r)}
  \Bigl[\Bigl(\partial_t+\partial_r\Bigr)(ru)\Bigr]^2\,dr\,d\omega\,dt
  \\+ \int_0^T\int_{\S^2}\int_0^\infty (1+r)^p \sigma'_U(t-r)
  e^{-\sigma_U(t-r)} \Bigl[\Bigl(\partial_t+\partial_r\Bigr)(ru)\Bigr]^2\,dr\,d\omega\,dt
  \\+ \Bigl(1-\frac{p}{2}\Bigr)\int_0^T\int_{\S^2}\int_0^\infty
  \frac{(r+1)^p}{r}e^{-\sigma_U(t-r)}|\ang(ru)|^2\,dr\,d\omega\,dt
\\+ \frac{p}{2}\int_0^T\int_{\S^2}\int_0^\infty
  \frac{(1+r)^{p-1}}{r}e^{-\sigma_U(t-r)}|\ang(ru)|^2\,dr\,d\omega\,dt. 
\end{multline}
Rearranging the terms, noting that
\[\sigma_U'(t-r)\gtrsim \frac{1}{\la t-r\ra},\quad \text{ on } X_U,\]
and taking a supremum over $T$ yields
\begin{multline*}
  \|\la r\ra^{\frac{p}{2}}
  r^{-1}(\partial_t+\partial_r)(ru)\|^2_{L^\infty L^2} + \|\la
  r\ra^{\frac{p}{2}}\ang u\|^2_{L^\infty L^2} + \|\la
  r\ra^{\frac{p-1}{2}} r^{-1}(\partial_t+\partial_r)(ru)\|^2_{L^2L^2}
  \\+ \|\la r\ra^{\frac{p}{2}} r^{-\frac{1}{2}} \ang u\|_{L^2L^2}^2
  + \sup_U \|\la r\ra^{\frac{p}{2}} \la t-r\ra^{-\frac{1}{2}}
  r^{-1}(\partial_t+\partial_r)(ru)\|^2_{L^2 L^2(X_U)}
 \\ \lesssim 
\|\la r\ra^{\frac{p}{2}}
  r^{-1}u(0,\cd)\|^2_{L^2} + 
\|\la
  r\ra^{\frac{p}{2}}\good u(0,\cd)\|^2_{L^2} + \int_0^\infty \int \la
  r\ra^p |\Box u| |r^{-1}(\partial_t+\partial_r)(ru)|\,dx\,dt.
\end{multline*}
By the Schwarz inequality, we may bound
\begin{multline}\label{forcing}
  \int_0^\infty \int \la
  r\ra^p |\Box u| |r^{-1}(\partial_t+\partial_r)(ru)|\,dx\,dt
  \lesssim \Bigl(\sum_\tau \sum_{R\le \tau/4} \|\la r\ra^{\frac{p+1}{2}}
  \Box u\|^2_{L^2L^2(C^R_\tau)}\Bigr)^{\frac{1}{2}} \|\la
  r\ra^{\frac{p-1}{2}} r^{-1}(\partial_t+\partial_r)(ru)\|_{L^2L^2}
  \\+ \Bigl\{\sum_U \Bigl[\sum_{\tau\ge 4U} U \|\la
  r\ra^{\frac{p}{2}} \Box
  u\|_{L^2L^2(C^U_\tau)}^2\Bigr]^{\frac{1}{2}}\Bigr\} \Bigl(\sup_U U^{-\frac{1}{2}}\|\la
  r\ra^{\frac{p}{2}} 
  r^{-1}(\partial_t+\partial_r)(ru)\|_{L^2L^2(X_U)}\Bigr).
\end{multline}
The second factor of each term may be bootstrapped.  Combining what
results with \eqref{Hardy} completes the proof.
\end{proof}

We next combine the previous proposition with a modification of the
$(\partial_t-\partial_r)$ portion of the multiplier in \eqref{multiplier}.  While the new
$\partial_t-\partial_r$ terms are easily controlled using the $LE^1$
norm, the corresponding forcing term comes with an added factor of
decay, $(1+r)^{-\delta}$, when compared to the right side of \eqref{LE}.
\begin{theorem}
  Fix $0<p<2$ and $\delta>0$.  If $u\in C^2(\R_+\times\R^3)$ and
$r^{\frac{p+2}{2}}|\partial^{\le 1} u(t,x)|\to 0$ as $|x|\to \infty$,
then
\begin{multline}
  \label{newLE}
  \|\la
  r\ra^{-\frac{\delta}{2}}(\partial_t-\partial_r)u\|_{L^\infty
    L^2} + \|\la r\ra^{\frac{p}{2}}
  (\partial_t+\partial_r)u\|_{L^\infty L^2} + \|\la
  r\ra^{\frac{p}{2}}\ang u\|_{L^\infty L^2}
+ \|\la r\ra^{\frac{p-1}{2}}r^{-\frac{1}{2}} u\|_{L^\infty L^2}
  \\+ \|\la r\ra^{-\frac{1+\delta}{2}}
 (\partial_t-\partial_r)u\|_{L^2L^2} + \|\la
  r\ra^{\frac{p-1}{2}} (\partial_t+\partial_r)u\|_{L^2L^2} +
  \|\la r\ra^{\frac{p-1}{2}}\ang u\|_{L^2L^2}
  +\|\la r\ra^{\frac{p-1}{2}}r^{-1} u\|_{L^2L^2}
  \\+ \sup_U U^{-\frac{1}{2}} \|\la r\ra^{\frac{p}{2}}
  r^{-1}(\partial_t+\partial_r)(ru)\|_{L^2L^2(X_U)}
  \lesssim
  \|\la
  r\ra^{-\frac{\delta}{2}}(\partial_t-\partial_r)u(0,\cd)\|_{L^2} \\+
  \|\la r\ra^{\frac{p}{2}} (\partial_t+\partial_r)u(0,\cd)\|_{L^2} +
\|\la r\ra^{\frac{p}{2}}\ang u(0,\cd)\|_{L^2}+
\|\la r\ra^{\frac{p}{2}}r^{-1} u(0,\cd)\|_{L^2}
\\+ \|\la r\ra^{\frac{1-\delta}{2}}\Box u\|_{L^2L^2}
+ \Bigl(\sum_\tau \sum_{R\le \tau/4} \|\la r\ra^{\frac{1+p}{2}} \Box
u\|^2_{L^2L^2(C^R_\tau)}\Bigr)^{\frac{1}{2}}
+\sum_U \Bigl(\sum_{\tau \ge 4U} U \|\la
r\ra^{\frac{p}{2}} \Box u\|^2_{L^2L^2(C^U_\tau)}\Bigr)^{\frac{1}{2}}.
\end{multline}
\end{theorem}

\begin{proof}
  Using \eqref{conjugate} and the related identity
\[\Bigl(\partial_t-\partial_r-\frac{1}{r}\Bigr)u = r^{-1}\Bigl(\partial_t-\partial_r\Bigr)(ru),\]
  we begin by considering
  \[\int_0^T\int (1+r)^{-\delta} \Box u
    \Bigl(\partial_t-\partial_r-\frac{1}{r}\Bigr)u\,dx\,dt
    = \int_0^T\int_{\S^2}\int_0^\infty
    (1+r)^{-\delta}\Bigl(\partial_t^2-\partial_r^2-\ang\cdot\ang\Bigr)(ru)
    \Bigl(\partial_t-\partial_r\Bigr)(ru)\,dr\,d\omega\,dt\]
  with $\delta>0$.
  Integrating by parts and using \eqref{angCommutator}, this is
  \begin{multline*}=\frac{1}{2}\int_0^T\int_{\S^2}\int_0^\infty
    (1+r)^{-\delta}\Bigl(\partial_t+\partial_r\Bigr)\Bigl[\Bigl(\partial_t-\partial_r\Bigr)
    (ru)\Bigr]^2\,dr\,d\omega\,dt
   -\int_0^T\int_{\S^2}\int_0^\infty
   \frac{(1+r)^{-\delta}}{r}|\ang(ru)|^2\,dr\,d\omega\,dt
   \\+\frac{1}{2} \int_0^T\int_{\S^2}\int_0^\infty
   (1+r)^{-\delta}\Bigl(\partial_t-\partial_r\Bigr)|\ang(ru)|^2\,dr\,d\omega\,dt.  
 \end{multline*}
The Fundamental Theorem of Calculus and subsequent integrations by
parts give
\begin{multline}\label{minus}
\int_0^T\int (1+r)^{-\delta} \Box u
    \Bigl(\partial_t-\partial_r-\frac{1}{r}\Bigr)u\,dx\,dt =
    \frac{1}{2}\int_{\S^2}\int_0^\infty 
  (1+r)^{-\delta}\Bigl\{\Bigl[\Bigl(\partial_t-\partial_r\Bigr)(ru)\Bigr]^2
  + |\ang (ru)|^2\Bigr\}\,dr\,d\omega\Bigl|_{t=0}^T
  \\- \frac{1}{2}\int_0^T \int_{\S^2} u^2(t,0)\,d\omega\,dt
  +\frac{\delta}{2}\int_0^T\int_{\S^2}\int_0^\infty
  (1+r)^{-1-\delta}\Bigl[\Bigl(\partial_t-\partial_r\Bigr)(ru)\Bigr]^2\,
  dr\,d\omega\,dt   \\-\int_0^T\int_{\S^2}\int_0^\infty
   (1+r)^{-\delta} \Bigl(r^{-1}+\frac{\delta}{2}
   (1+r)^{-1}\Bigr)|\ang(ru)|^2\,dr\,d\omega\,dt. 
\end{multline}

We now consider a multiplier of the form
\[(1+r)^{-\delta}\Bigl(\partial_t-\partial_r-\frac{1}{r}\Bigr) +
  C(1+r)^p
  e^{-\sigma_U(t-r)}\Bigl(\partial_t+\partial_r+\frac{1}{r}\Bigr),\quad
  C\gg 1,\]
by adding a large multiple of \eqref{plus} to \eqref{minus}.  Since
$\sigma_U$ is bounded independently of $U$, for a sufficiently large
$C$,
\[\frac{C}{2}\int_0^T\int_{\S^2} e^{-\sigma_U(t)}
  u^2(t,0)\,d\omega\,dt - \frac{1}{2}\int_0^T\int_{\S^2}
  u^2(t,0)\,d\omega\,dt \ge 0,\]
and as such, this $r=0$ boundary term may be dropped.  The nonnegative
contribution
\[\frac{1}{2}\int_{\S^2}\int_0^\infty (1+r)^{-\delta}
  |\ang(ru)(T,r\omega)|^2\,dr\,d\omega\]
may also be omitted, and since $\delta, p > 0$, we can simplify by
bounding
\[\frac{1}{2}\int_{\S^2}\int_0^\infty (1+r)^{-\delta}
  |\ang(ru)(0,r\omega)|^2\,dr\,d\omega
\le \frac{1}{2}\int_{\S^2}\int_0^\infty (1+r)^{p}
  |\ang(ru)(0,r\omega)|^2\,dr\,d\omega.
\]
What then results from this combination of \eqref{minus} and
\eqref{plus} is
\begin{multline}\label{newLE1}
  \Bigl\|(1+r)^{-\frac{\delta}{2}}r^{-1}
 \Bigl(\partial_t-\partial_r\Bigr)(ru)\Bigl\|^2_{L^\infty L^2}
+\Bigl\|(1+r)^{\frac{p}{2}}
r^{-1}\Bigl(\partial_t+\partial_r\Bigr)(ru)\Bigr\|^2_{L^\infty L^2}
+ \Bigl\|(1+r)^{\frac{p}{2}}\ang u\Bigr\|^2_{L^\infty L^2} \\+\Bigl\|(1+r)^{-\frac{1}{2}-\frac{\delta}{2}}
r^{-1}\Bigl(\partial_t-\partial_r\Bigr)(ru)\Bigr\|^2_{L^2 L^2}
+ \Bigl\|(1+r)^{\frac{p-1}{2}}
r^{-1}\Bigl(\partial_t+\partial_r\Bigr)(ru)\Bigr\|^2_{L^2 L^2}
+\Bigl\|(1+r)^{\frac{p-1}{2}} \ang u\Bigr\|^2_{L^2 L^2}
\\+ \sup_U U^{-1} \Bigl\|(1+r)^{\frac{p}{2}} r^{-1}\Bigl(\partial_t+\partial_r\Bigr)(ru)\Bigr\|^2_{L^2L^2(X_U)}
  \\\lesssim
\Bigl\|(1+r)^{-\frac{\delta}{2}} r^{-1}
\Bigl(\partial_t-\partial_r\Bigr)(ru)(0,\cd)\Bigr\|^2_{L^2} 
+\Bigl\|(1+r)^{\frac{p}{2}}
r^{-1}\Bigr(\partial_t+\partial_r\Bigr)(ru)(0,\cd)\Bigr\|^2_{L^2}
+ \Bigl\|(1+r)^{\frac{p}{2}} \ang u(0,\cd)\Bigr\|^2_{L^2}
\\+  \int_0^\infty\int |\Box
  u|\Bigl\{(1+r)^{-\delta}\Bigl|\Bigl(\partial_t-\partial_r-\frac{1}{r}\Bigr)u\Bigr|
  + (1+r)^p \Bigl|\Bigl(\partial_t+\partial_r+\frac{1}{r}\Bigr)u\Bigr|\Bigr\}\,dx\,dt.
\end{multline}
We use \eqref{forcing} and the
fact that the Schwarz inequality allows us to bound
\[
  \int_0^\infty\int |\Box
  u|
  (1+r)^{-\delta}\Bigl|\Bigl(\partial_t-\partial_r-\frac{1}{r}\Bigr)u\Bigr|\,dx\,dt
  \lesssim \|(1+r)^{\frac{1-\delta}{2}} \Box u\|_{L^2L^2}
  \|(1+r)^{-\frac{1+\delta}{2}} r^{-1}(\partial_t-\partial_r)(ru)\|_{L^2L^2}.
\]
Bootstrapping then gives
\begin{multline*}
  \|\la
  r\ra^{-\frac{\delta}{2}}r^{-1}(\partial_t-\partial_r)(ru)\|_{L^\infty
    L^2} + \|\la r\ra^{\frac{p}{2}}
  r^{-1}(\partial_t+\partial_r)(ru)\|_{L^\infty L^2} + \|\la
  r\ra^{\frac{p}{2}}\ang u\|_{L^\infty L^2}
  \\+ \|\la r\ra^{-\frac{1+\delta}{2}}
  r^{-1}(\partial_t-\partial_r)(ru)\|_{L^2L^2} + \|\la
  r\ra^{\frac{p-1}{2}} r^{-1}(\partial_t+\partial_r)(ru)\|_{L^2L^2} +
  \|\la r\ra^{\frac{p-1}{2}}\ang u\|_{L^2L^2}
  \\+ \sup_U U^{-\frac{1}{2}} \|\la r\ra^{\frac{p}{2}}
  r^{-1}(\partial_t+\partial_r)(ru)\|_{L^2L^2(X_U)}
  \lesssim
  \|\la
  r\ra^{-\frac{\delta}{2}}(\partial_t-\partial_r)u(0,\cd)\|_{L^2} \\+
  \|\la r\ra^{\frac{p}{2}} (\partial_t+\partial_r)u(0,\cd)\|_{L^2} +
\|\la r\ra^{\frac{p}{2}}\ang u(0,\cd)\|_{L^2}+
\|\la r\ra^{\frac{p}{2}}r^{-1} u(0,\cd)\|_{L^2}
\\+ \|\la r\ra^{\frac{1-\delta}{2}}\Box u\|_{L^2L^2}
+ \Bigl(\sum_\tau \sum_{R\le \tau/4} \|\la r\ra^{\frac{1+p}{2}} \Box
u\|^2_{L^2L^2(C^R_\tau)}\Bigr)^{\frac{1}{2}}
+\sum_U \Bigl(\sum_{\tau \ge 4U} U \|\la
r\ra^{\frac{p}{2}} \Box u\|^2_{L^2L^2(C^U_\tau)}\Bigr)^{\frac{1}{2}}.
\end{multline*}
Pairing this with \eqref{Hardy} then completes the proof.
\end{proof}


\section{Sobolev Estimates}
In this section, we collect our principal decay estimates, which are
variants of the Klainerman-Sobolev estimate \cite{klainermanSob}.

On occasion, it will suffice to apply the following standard weighted Sobolev
estimate, which is also from \cite{klainermanSob} and follows by applying Sobolev
embeddings in the $r$ and $\omega$ variables after localizing.
\begin{lemma}
For $h\in
  C^\infty(\R^3)$ and $R>0$,
  \begin{equation}
    \label{weightedSob}
    \|h\|_{L^\infty(A_R)} \lesssim R^{-1} \|Z^{\le 2} h\|_{L^2(\tilde{A}_R)}.
  \end{equation}
\end{lemma}

Where a finer analysis is necessary, we shall use the space-time
Klainerman-Sobolev estimates of \cite[Lemma 3.8]{MTT}.  We record
these in the following lemma.
\begin{lemma}
If $\tau \ge 1$, $1\le R\le\tau/2$, and $1\le U\le \tau/4$, then
  \begin{align}
    \label{MTT_R}
    \|w\|_{L^\infty L^\infty(C^R_\tau)}&\lesssim \frac{1}{\tau^{\frac{1}{2}}R^{\frac{3}{2}}}\|Z^{\le 2}
                                         w\|_{L^2L^2(\tilde{C}^R_\tau)}
                                         + \frac{1}{\tau^{\frac{1}{2}}
                                         R^{\frac{1}{2}}}
                                         \|\partial_r
                                         Z^{\le 2}
                                         w\|_{L^2L^2(\tilde{C}^R_\tau)}, 
                                       \\
    \label{MTT_U}
  \|w\|_{L^\infty L^\infty(C^U_\tau)} &\lesssim
  \frac{1}{\tau^{\frac{3}{2}}U^{\frac{1}{2}}} \|Z^{\le 2} w\|_{L^2L^2(\tilde{C}^U_\tau)} +
  \frac{U^{\frac{1}{2}}}{\tau^{\frac{3}{2}}} \|\partial_r Z^{\le 2} w\|_{L^2L^2(\tilde{C}^U_\tau)}.
\end{align}
\end{lemma}

We shall only tersely describe the proof since this result previously
appeared in \cite{MTT}.
If $R=1$, \eqref{MTT_R} is an immediate consequence of standard
Sobolev embeddings.  And if $1<R\lesssim \tau$, then after localizing,
we may apply Sobolev embeddings in $(s,\omega)$ and the 
Fundamental Theorem of Calculus in $\rho$ where
$t=e^s$ and $r=e^{s+\rho}$.  This gives that
\[\beta\Bigl(\frac{e^s}{\tau}\Bigr)\beta(\frac{e^{s+\rho}}{R}\Bigr) |w(e^s,
  e^{s+\rho}\omega)|
  \lesssim \Bigl(\int \int\int \Bigl|\partial_\rho
  \bigl|\partial_{s,\omega}^{\le 2}
  \Bigl[\beta(\frac{e^s}{\tau}\Bigr)\beta(\frac{e^{s+\rho}}{R}\Bigr)
  w(e^s,
  e^{s+\rho}\omega)\Bigr]\bigr|^2\Bigl|\,d\rho\,ds\,d\omega\Bigr)^{1/2}.\] 
Relying on the observations that
\[\partial_s (w(e^s, e^{s+\rho}\omega)) = (Sw)(e^s, e^{s+\rho}
  \omega),\quad |\partial_\omega (w(e^s,e^{s+\rho}\omega)|\lesssim
  |(\Omega w)(e^s, e^{s+\rho}\omega)|,\]
  \[\partial_\rho (w(e^s,
    e^{s+\rho}\omega)) = (r\partial_r w)(e^s, e^{s+\rho}\omega),\]
upon changing variables in the integrals, we see that
\begin{equation}
  \label{MTT_R2}
  \|w\|_{L^\infty L^\infty(C^R_\tau)} \lesssim
  \frac{1}{\tau^{\frac{1}{2}}R^{\frac{3}{2}}} \|Z^{\le 2}
  w\|_{L^2L^2(\tilde{C}^R_\tau)} + \frac{1}{\tau^{\frac{1}{2}}R}
  \|Z^{\le 2} w\|^{\frac{1}{2}}_{L^2L^2(\tilde{C}^R_\tau)} \|\partial_r
  Z^{\le 2} w\|^{\frac{1}{2}}_{L^2L^2(\tilde{C}^R_\tau)}.
\end{equation}
The estimate \eqref{MTT_R} is now an immediate consequence.  Moreover,
if $R=\tau/2$, if we replace $w$ by $\chi\Bigl(\frac{2(t-r)}{\tau}\Bigr)
w$ and note that
$S\Bigl(\chi\Bigl(\frac{2(t-r)}{\tau}\Bigr)\Bigr)=\O(1)$, the estimate
for $C^{\tau/2}_\tau$ also follows.

When $U=1$, the bound \eqref{MTT_U} follows from \eqref{MTT_R2}.
Otherwise, with $t=e^s$ and $t-r=e^{s+\rho}$, a similar application
of Sobolev embeddings yields \eqref{MTT_U}.

When the estimates of \cite{MTT} are applied to $\partial w$, the
decomposition pairs nicely with 
\begin{equation}
  \label{S}
 2Sw=(t+r)(\partial_t+\partial_r)w+(t-r)(\partial_t-\partial_r)w,
\end{equation}
which will allow us to recover $\Box$ to get additional decay out of
the second derivative terms.  This represents space-time analogs of
some estimates of \cite{KlSid}.  See, also, \cite{Lindblad-Tohaneanu}
where some similar analyses appeared previously.

\begin{corollary} For $\tau \ge 1$ and $1\le R\le \tau/2$, $1\le U\le \tau/4$, we have
\begin{align}
  \label{CRT}
 \|\partial w\|_{L^\infty L^\infty(C^R_\tau)} &\lesssim \frac{1}{\tau^{\frac{1}{2}}R^{\frac{3}{2}}}
 \|\partial Z^{\le 3} w\|_{L^2L^2(\tC^R_\tau)} + \frac{1}{\tau^{\frac{1}{2}}R^{\frac{1}{2}}}\|\Box
 Z^{\le 2} w\|_{L^2L^2(\tC^R_\tau)},\\
  \label{CUT}
 \|\partial w\|_{L^\infty L^\infty(C^U_\tau)} &\lesssim \frac{1}{U^{\frac{1}{2}}\tau^{\frac{3}{2}}}
 \|\partial Z^{\le 3} w\|_{L^2L^2(\tC^U_\tau)} +
 \frac{1}{U^{\frac{1}{2}} \tau^{\frac{1}{2}}}\|\Box
 Z^{\le 2} w\|_{L^2L^2(\tC^U_\tau)}.
\end{align}
\end{corollary}

\begin{proof}
  We apply \eqref{MTT_R} to see
  \begin{equation}\label{crt1}\|\partial w\|_{L^\infty
      L^\infty(C^R_\tau)} \lesssim \frac{1}{\tau^{\frac{1}{2}}R^{\frac{3}{2}}} 
 \|\partial Z^{\le 2} w\|_{L^2L^2(\tC^R_\tau)} +
 \frac{1}{\tau^{\frac{1}{2}}R^{\frac{1}{2}}}\|\partial_r \partial
 Z^{\le 2} w\|_{L^2L^2(\tC^R_\tau)}.
\end{equation}
We notice that
\begin{equation}\label{angcrt}\|\partial_r \ang Z^{\le 2} w\|_{L^2L^2(\tC^R_\tau)} \lesssim
  \frac{1}{R}\|\partial Z^{\le 3} w\|_{L^2L^2(\tC^R_\tau)}\end{equation}
follows from \eqref{angCommutator} and \eqref{angBound}.
Moreover, by applying \eqref{S} with $w$ replaced by
$(\partial_t-\partial_r)Z^{\le 2}w$ and $(\partial_t+\partial_r)Z^{\le
  2} w$ respectively, we obtain
\begin{equation}\label{minuscrt}\|(\partial_t-\partial_r)^2 Z^{\le 2} w\|_{L^2L^2(\tC^R_\tau)}
  \lesssim \frac{1}{R}\|\partial Z^{\le 3} w\|_{L^2L^2(\tC^R_\tau)} +
  \|(\partial_t^2-\partial_r^2) Z^{\le 2} w\|_{L^2L^2(\tC^R_\tau)},
\end{equation}
and
\begin{equation}\label{pluscrt}\|(\partial_t+\partial_r)^2 Z^{\le 2} w\|_{L^2L^2(\tC^R_\tau)}
  \lesssim \frac{1}{R}\|\partial Z^{\le 3} w\|_{L^2L^2(\tC^R_\tau)} +
  \|(\partial_t^2-\partial_r^2) Z^{\le 2}w\|_{L^2L^2(\tC^R_\tau)}.
\end{equation}

Observing that
\[|\partial_r \partial Z^{\le 2} w|\le |\partial_r \ang Z^{\le 2} w| +
  |\partial_r (\partial_t-\partial_r)Z^{\le 2} w| +
  |\partial_r(\partial_t+\partial_r) Z^{\le 2} w|\]
and subsequently writing $\partial_r =
\frac{1}{2}\Bigl[(\partial_t+\partial_r) -
(\partial_t-\partial_r)\Bigr]$ in the last two terms, we see that
\begin{multline}\label{drdsplit}
  \|\partial_r \partial Z^{\le 2}w\|_{L^2L^2(\tC^R_\tau)}
  \lesssim \|\partial_r \ang Z^{\le 2}w\|_{L^2L^2(\tC^R_\tau)} +
  \|(\partial_t-\partial_r)^2 Z^{\le 2} w\|_{L^2L^2(\tC^R_\tau)} \\+
  \|(\partial_t+\partial_r)^2 Z^{\le 2}w\|_{L^2L^2(\tC^R_\tau)} +
  \|(\partial_t^2-\partial_r^2)Z^{\le 2}w\|_{L^2L^2(\tC^R_\tau)}.
\end{multline}
Using this in \eqref{crt1} and the estimating via \eqref{angcrt}, \eqref{minuscrt}, and \eqref{pluscrt} gives
\begin{equation*}\|\partial w\|_{L^\infty L^\infty(C^R_\tau)} \lesssim \frac{1}{\tau^{\frac{1}{2}}R^{\frac{3}{2}}}
 \|\partial Z^{\le 3} w\|_{L^2L^2(\tC^R_\tau)} +
 \frac{1}{\tau^{\frac{1}{2}} R^{\frac{1}{2}}}\|(\partial_t^2-\partial_r^2)
 Z^{\le 2} w\|_{L^2L^2(\tC^R_\tau)}.
\end{equation*}
Relying upon
\begin{equation}\label{boxA}\partial_t^2-\partial_r^2 = \Box +\frac{2}{r}\partial_r +
\ang\cdot\ang
\end{equation}
and \eqref{angBound}
yields
\eqref{CRT}.

For \eqref{CUT}, we argue similarly.  By \eqref{angCommutator} and
\eqref{angBound} (applied to $\partial_r Z^{\le 2} w$), we obtain
\[\|\partial_r \ang Z^{\le 2} w\|_{L^2L^2(\tC^U_\tau)}
  \lesssim \frac{1}{\tau} \|\partial Z^{\le 3}
  w\|_{L^2L^2(\tC^U_\tau)}.\]
And using \eqref{S} with $w$ replaced by
$(\partial_t-\partial_r)Z^{\le 2}w$ and $(\partial_r+\partial_r)Z^{\le
  2} w$ respectively, we see 
that
\begin{align*}
\|(\partial_t-\partial_r)^2 Z^{\le 2}
  w\|_{L^2L^2(\tC^U_\tau)}&\lesssim \frac{1}{U}\|\partial Z^{\le
                            3}w\|_{L^2L^2(\tC^U_\tau)} +
                            \frac{\tau}{U}
                            \|(\partial_t^2-\partial_r^2) Z^{\le 2}
                            w\|_{L^2L^2(\tC^U_\tau)},\\
  \|(\partial_t+\partial_r)^2 Z^{\le 2}
  w\|_{L^2L^2(\tC^U_\tau)}&\lesssim \frac{1}{\tau}\|\partial Z^{\le
                            3}w\|_{L^2L^2(\tC^U_\tau)} +
                            \|(\partial_t^2-\partial_r^2) Z^{\le 2} w\|_{L^2L^2(\tC^U_\tau)}.
\end{align*}
Using these in \eqref{MTT_U} and the $\tC^U_\tau$ analog of \eqref{drdsplit}, in combination with \eqref{boxA} and
\eqref{angBound} as above, we see that 
\begin{align*}
  \|\partial w\|_{L^\infty L^\infty (C^U_\tau)} &\lesssim
                                         \frac{1}{U^{\frac{1}{2}}\tau^{\frac{3}{2}}}
                                         \|\partial Z^{\le 2}
                                         w\|_{L^2L^2(\tC^U_\tau)} +
                                         \frac{U^{\frac{1}{2}}}{\tau^{\frac{3}{2}}}
                                         \|\partial_r\partial
                                         Z^{\le 2}
                                         w\|_{L^2L^2(\tC^U_\tau)}\\
  &\lesssim  \frac{1}{U^{\frac{1}{2}}\tau^{\frac{3}{2}}}
                                         \|\partial Z^{\le 3}
                                         w\|_{L^2L^2(\tC^U_\tau)} +
                                         \frac{1}{U^{\frac{1}{2}} \tau^{\frac{1}{2}}}
                                         \|\Box
                                         Z^{\le 2}
                                         w\|_{L^2L^2(\tC^U_\tau)},
\end{align*}
which completes the proof.
\end{proof}


\section{Global Existence}
Here we provide the proof of Theorem~\ref{main_theorem}.  To do so, we
set  $u_0\equiv 0$, $v_0\equiv 0$ and recursively define $u_k, v_k$ to solve
\begin{equation}
  \label{iteration_equation}
  \begin{cases}
    \Box u_k = (\partial_t+\partial_r) u_{k-1} \,\partial_t v_{k-1}
-\partial_r u_{k-1}(\partial_t+\partial_r)v_{k-1}
    - \ang u_{k-1}\cdot\ang v_{k-1},\\ 
    \Box v_k = \partial_t u_{k-1} \,\partial_t v_{k-1},\\
     (u_k(0,\cd),\partial_tu_k(0,\cd))=(u_{(0)},u_{(1)}),\\
    (v_k(0,\cd),\partial_t v_k(0,\cd))=(v_{(0)},v_{(1)}).
  \end{cases}
\end{equation}
We will show that the sequences $(u_k)$ and $(v_k)$ converge.  The
limits yield the desired solutions $u, v$ to \eqref{main_equation}.

\subsection{Boundedness}  We fix $0<p<1$, $0<\delta < \min\Bigl(p,
1-p\Bigr)$, and $N$ large enough so that
$\frac{N}{2}+3\le N$.  We then 
set 
\begin{multline}
  \label{M}
  M_k =
  \|\la r\ra^{\frac{p-1}{2}} \good Z^{\le N}u_{k}\|_{L^2 L^2} +
  \|\la r\ra^{\frac{p-1}{2}} r^{-1} Z^{\le N} u_{k}\|_{L^2L^2}
+
  \|\la r\ra^{\frac{p-1}{2}} \good Z^{\le N} v_{k}\|_{L^2L^2} +
   \|\la
  r\ra^{\frac{p-1}{2}} r^{-1} Z^{\le N}v_{k}\|_{L^2L^2}
  \\+ \|Z^{\le N}u_{k}\|_{LE^1} + \|\partial Z^{\le
    N}u_{k}\|_{L^\infty L^2}
  + \|\la r\ra^{-\frac{1+\delta}{2}} \partial Z^{\le N}
  v_{k}\|_{L^2 L^2}
  + \|\la r \ra^{-\frac{\delta}{2}} \partial Z^{\le
    N} v_k\|_{L^\infty L^2} 
  \\+\sup_\tau \sup_{R\le \tau/2}
    \Bigl(\tau^{\frac{1}{2}}R \|\partial Z^{\le \frac{N}{2}}
    u_k\|_{L^\infty 
      L^\infty(C^R_\tau)}\Bigr)
  +\Bigl[\sum_\tau \sum_{R\le \tau/2}
    \Bigl(\tau^{\frac{1}{2}} R^{1-\frac{\delta}{2}} \|\partial Z^{\le \frac{N}{2}}
    v_k\|_{L^\infty 
      L^\infty(C^R_\tau)}\Bigr)^2\Bigr]^{\frac{1}{2}}
  \\+\sup_\tau \sup_{U\le \tau/4}
    \Bigl(\tau U^{\frac{1}{2}}\|\partial Z^{\le \frac{N}{2}}
    u_k\|_{L^\infty 
      L^\infty(C^U_\tau)}\Bigr)
  +\Bigl[\sum_\tau \sum_{U\le \tau/4}
    \Bigl(\tau^{1-\frac{\delta}{2}} U^{\frac{1}{2}}\|\partial Z^{\le \frac{N}{2}}
    v_k\|_{L^\infty 
      L^\infty(C^U_\tau)}\Bigr)^2\Bigr]^{\frac{1}{2}}.
\end{multline}
For any $k\ge 1$, we shall show that
\begin{equation}\label{MkGoal}M_k \le C_0\varepsilon + C M_{k-1}^2
\end{equation}
for some fixed constant $C_0$.  Provided that $\varepsilon>0$ is
sufficiently small, a straightforward induction argument then shows
that
\begin{equation}
  \label{MkBound}
  M_k\le 2 C_0\varepsilon
\end{equation}
for any $k$.

The product rule gives
\begin{equation}
  \label{product_rule_u}
  |Z^{\le N} \Box u_k| \lesssim |\partial Z^{\le\frac{N}{2}}  u_{k-1}|
  |\partial Z^{\le N} v_{k-1}| + |\partial Z^{\le N} u_{k-1}| |\partial Z^{\le
    \frac{N}{2}} v_{k-1}|.
\end{equation}
Hence,
\begin{multline}
  \label{boxu_crt}
  \|\la r\ra^{\frac{p+1}{2}}\Box Z^{\le N} u_k\|_{L^2L^2(C^R_\tau)}
  \lesssim   \tau^{-\frac{1}{2}} R^{\frac{p+\delta}{2}}\Bigl(\tau^{\frac{1}{2}}R \|\partial Z^{\le
    \frac{N}{2}} u_{k-1}\|_{L^\infty L^\infty(C^R_\tau)}\Bigr)
  \|\la r\ra^{-\frac{1+\delta}{2}} \partial Z^{\le N}
  v_{k-1}\|_{L^2L^2(C^R_\tau)}
  \\+ \tau^{-\frac{1}{2}}R^{\frac{p+\delta}{2}}\Bigl(\tau^{\frac{1}{2}}R^{1-\frac{\delta}{2}} \|\partial Z^{\le
    \frac{N}{2}} v_{k-1}\|_{L^\infty L^\infty(C^R_\tau)}\Bigr)
  \|Z^{\le N}
  u_{k-1}\|_{LE^1},
\end{multline}
and
\begin{multline}
  \label{boxu_cut}
 U^{\frac{1}{2}} \|\la r\ra^{\frac{p}{2}}\Box Z^{\le N} u_k\|_{L^2L^2(C^U_\tau)}
  \lesssim   \tau^{\frac{p-1+\delta}{2}}\Bigl(U^{\frac{1}{2}}\tau \|\partial Z^{\le
    \frac{N}{2}} u_{k-1}\|_{L^\infty L^\infty(C^U_\tau)}\Bigr)
  \|\la r\ra^{-\frac{1+\delta}{2}} \partial Z^{\le N}
  v_{k-1}\|_{L^2L^2(C^U_\tau)}
  \\+ \tau^{\frac{p-1+\delta}{2}}\Bigl(U^{\frac{1}{2}}\tau^{1-\frac{\delta}{2}} \|\partial Z^{\le
    \frac{N}{2}} v_{k-1}\|_{L^\infty L^\infty(C^U_\tau)}\Bigr)
  \|Z^{\le N}
  u_{k-1}\|_{LE^1}.
\end{multline}
From this, it follows that
\begin{equation}
  \label{summed_boxu}
  \Bigl(\sum_\tau \sum_{R\le \tau/2} \|\la r\ra^{\frac{1+p}{2}} \Box
  Z^{\le N}
u_k\|^2_{L^2L^2(C^R_\tau)}\Bigr)^{\frac{1}{2}}
+\sum_U \Bigl(\sum_{\tau \ge 4U} U \|\la
r\ra^{\frac{p}{2}} \Box Z^{\le N} u_k\|^2_{L^2L^2(C^U_\tau)}\Bigr)^{\frac{1}{2}}
\lesssim M_{k-1}^2.
\end{equation}
By \eqref{MR} and \eqref{smallness}, along with a Hardy inequality, we get
\begin{equation}\label{Ik}
   \|\la r\ra^{\frac{p-1}{2}} \good Z^{\le N}u_{k}\|_{L^2  L^2} +
  \|\la r\ra^{\frac{p-1}{2}}r^{-1} Z^{\le N} u_{k}\|_{L^2 L^2} \lesssim
                                                              \varepsilon
                                                              +  M_{k-1}^2.
\end{equation}

As the above argument does not rely on the null structure of $\Box
u_k$ and we also have
\begin{equation}
  \label{product_rule_v}
  |Z^{\le N} \Box v_k| \lesssim |\partial Z^{\le\frac{N}{2}} u_{k-1}|
  |\partial Z^{\le N} v_{k-1}| + |\partial Z^{\le N} u_{k-1}|
  |\partial Z^{\le
    \frac{N}{2}} v_{k-1}|,
\end{equation}
the same arguments show that
\begin{equation}
  \label{IIIk}
  \Bigl(\sum_\tau \sum_{R\le \tau/2} \|\la r\ra^{\frac{1+p}{2}} \Box
  Z^{\le N}
v_k\|^2_{L^2L^2(C^R_\tau)}\Bigr)^{\frac{1}{2}}
+\sum_U \Bigl(\sum_{\tau \ge 4U} U \|\la
r\ra^{\frac{p}{2}} \Box Z^{\le N} v_k\|^2_{L^2L^2(C^U_\tau)}\Bigr)^{\frac{1}{2}}
\lesssim M_{k-1}^2,
\end{equation}
which we shall use later.

It is in the process of bounding $\|Z^{\le N} u_k\|_{LE^1}$ and $\|\partial Z^{\le N}
u_k\|_{L^\infty L^2}$ that we will need the null condition.  A finer
alternative to \eqref{product_rule_u} that takes care with the good
directions is
\begin{multline}
  \label{product_rule_u_null}
  |\Box Z^{\le N} u_k| \lesssim |Z^{\le \frac{N}{2}} \partial
  v_{k-1}|
  |Z^{\le N} \good u_{k-1}| 
  +|Z^{\le \frac{N}{2}} \good u_{k-1}| |Z^{\le N} \partial v_{k-1}|
  \\+|Z^{\le \frac{N}{2}}\partial u_{k-1}|
  |Z^{\le N} \good v_{k-1}| 
  +| Z^{\le \frac{N}{2}} \good v_{k-1}| |Z^{\le N} \partial u_{k-1}|.
\end{multline}

We need to consider
\[\int_0^\infty \|Z^{\le N} \Box u_k(s,\cd)\|_{L^2}\,ds \le \sum_{j\ge
    0} \int_0^\infty \|Z^{\le N} \Box u_k(s,\cd)\|_{L^2(A_{2^j})}\,ds.\]
To each lower order term in \eqref{product_rule_u_null}, we apply
\eqref{weightedSob} to see that this is
\begin{multline*}
 \lesssim \sum_{j\ge 0} \int_0^\infty 2^{-\frac{p-\delta}{2}j}\|\la
  r\ra^{\frac{p-1}{2}}Z^{\le N}\good 
  u_{k-1}(s,\cd)\|_{L^2(A_{2^j})} \|\la
  r\ra^{-\frac{1+\delta}{2}} Z^{\le \frac{N}{2}+2} \partial v_{k-1}(s,\cd) \|_{L^2(\tilde{A}_{2^j})}
  \,ds
\\+\sum_{j\ge 0} \int_0^\infty 2^{-\frac{p-\delta}{2}j}\|\la
  r\ra^{\frac{p-1}{2}}Z^{\le \frac{N}{2}+2}\good 
  u_{k-1}(s,\cd)\|_{L^2(\tilde{A}_{2^j})} \|\la
  r\ra^{-\frac{1+\delta}{2}} Z^{\le N} \partial v_{k-1}(s,\cd) \|_{L^2(A_{2^j})}
  \,ds
  \\+\sum_{j\ge 0} \int_0^\infty 2^{-\frac{p}{2}j}\|\la
  r\ra^{-\frac{1}{2}} Z^{\le \frac{N}{2}+2}\partial 
  u_{k-1}(s,\cd)\|_{L^2(\tilde{A}_{2^j})} \|\la
  r\ra^{\frac{p-1}{2}}Z^{\le N} \good v_{k-1}(s,\cd) \|_{L^2(A_{2^j})}
  \,ds
  \\+\sum_{j\ge 0} \int_0^\infty 2^{-\frac{p}{2}j}\|\la
  r\ra^{-\frac{1}{2}} Z^{\le N} \partial
  u_{k-1}(s,\cd)\|_{L^2(A_{2^j})} \|\la
  r\ra^{\frac{p-1}{2}} Z^{\le \frac{N}{2}+2}\good v_{k-1}(s,\cd) \|_{L^2(\tilde{A}_{2^j})}
  \,ds.
\end{multline*}
By the Schwarz inequality and \eqref{commutators}, this is
\begin{multline*}
  \lesssim \Bigl(\|\la r\ra^{\frac{p-1}{2}}\good Z^{\le N}
  u_{k-1}\|_{L^2L^2} + \|\la r\ra^{\frac{p-1}{2}} r^{-1} Z^{\le N}
  u_{k-1}\|_{L^2L^2}\Bigr) \|\la r\ra^{-\frac{1+\delta}{2}} \partial
  Z^{\le N}v_{k-1}\|_{L^2L^2} \\+ \|Z^{\le N}
  u_{k-1}\|_{LE^1}\Bigl(\|\la r\ra^{\frac{p-1}{2}} \good Z^{\le N}
  v_{k-1}\|_{L^2L^2} + \|\la r\ra^{\frac{p-1}{2}} r^{-1} Z^{\le N}
  v_{k-1}\|_{L^2L^2}\Bigr)
\end{multline*}
provided $\frac{N}{2}+2\le N$, which is clearly $\O(M_{k-1}^2)$ as desired.
Hence, due to \eqref{LE1} and \eqref{smallness}, we have shown
\begin{equation}
  \label{Vk}
  \|Z^{\le N}u_k\|_{LE^1} + \|\partial Z^{\le N}u_k\|_{L^\infty L^2}
  \lesssim \varepsilon + M_{k-1}^2.
\end{equation}

In order to address the $v_k$ terms,
we next consider
\[\|\la r\ra^{\frac{1-\delta}{2}}Z^{\le N}\Box v_k\|_{L^2L^2}.
\]
To the lower order factors in \eqref{product_rule_v} we apply
\eqref{weightedSob} to see that this is
\[
  \lesssim \|\partial Z^{\le \frac{N}{2}+2} u_{k-1}\|_{L^\infty L^2}
  \|\la r\ra^{-\frac{1+\delta}{2}} \partial Z^{\le N}
  v_{k-1}\|_{L^2L^2}
    + \|\partial Z^{\le N} u_{k-1}\|_{L^\infty L^2}
  \|\la r\ra^{-\frac{1+\delta}{2}} \partial Z^{\le \frac{N}{2}+2}
  v_{k-1}\|_{L^2L^2}.\]
Since $\frac{N}{2}+2\le N$, this is $\O(M_{k-1}^2)$.  When combined
with \eqref{smallness}, \eqref{newLE}, \eqref{IIIk}, and the
observation that
\[\|\la
  r\ra^{-\frac{\delta}{2}} \partial Z^{\le N} v_k\|_{L^\infty L^2}
\lesssim \|\la
  r\ra^{-\frac{\delta}{2}} (\partial_t-\partial_r) Z^{\le N} v_k\|_{L^\infty L^2} + \|\la
  r\ra^{\frac{p}{2}} \good Z^{\le N} v_k\|_{L^\infty L^2},
  \]
this gives
\begin{multline}
  \label{VIk}
  \|\la r\ra^{\frac{p-1}{2}}\good Z^{\le N}v_k\|_{L^2L^2} + \|\la
  r\ra^{\frac{p-1}{2}} r^{-1}Z^{\le N} v_k\|_{L^2L^2} + \|\la
  r\ra^{-\frac{1+\delta}{2}} \partial Z^{\le N} v_k\|_{L^2L^2} \\+ \|\la
  r\ra^{-\frac{\delta}{2}} \partial Z^{\le N} v_k\|_{L^\infty L^2}
  \lesssim \varepsilon + M_{k-1}^2.
\end{multline}

In order to show \eqref{MkGoal}, it remains to bound the $L^\infty
L^\infty$ terms in \eqref{M}.
%
%
%
%
Applying \eqref{weightedSob} to each lower order piece in
\eqref{product_rule_u_null}, we see that 
\begin{multline*}
 R^{\frac{1}{2}} \|Z^{\le N} \Box u_k\|_{L^2L^2(\tC^R_\tau)} \lesssim R^{\frac{\delta-p}{2}}
  \|\la r\ra^{-\frac{\delta}{2}} Z^{\le \frac{N}{2}+2}\partial
  v_{k-1}\|_{L^\infty L^2} \|\la r\ra^{\frac{p-1}{2}} Z^{\le
    N} \good u_{k-1}\|_{L^2L^2}
\\+ R^{\frac{\delta-p}{2}} \|\la r\ra^{\frac{p-1}{2}} Z^{\le \frac{N}{2}+2} \good
u_{k-1}\|_{L^2L^2} \|\la r\ra^{-\frac{\delta}{2}} Z^{\le
  N} \partial v_{k-1}\|_{L^\infty L^2}
\\+ R^{-\frac{p}{2}} \|Z^{\le \frac{N}{2}+2} \partial
u_{k-1}\|_{L^\infty L^2} \|\la r\ra^{\frac{p-1}{2}} Z^{\le N}\good
v_{k-1}\|_{L^2L^2}
\\+ R^{-\frac{p}{2}} \|Z^{\le N} \partial
u_{k-1}\|_{L^\infty L^2} \|\la r\ra^{\frac{p-1}{2}} Z^{\le \frac{N}{2}+2}\good v_{k-1}\|_{L^2L^2}.
\end{multline*}
And thus, by \eqref{CRT} and the facts that $\frac{N}{2}+2\le N$ and
$0<\delta<p$,
\begin{multline*}
  \tau^{\frac{1}{2}} R \|\partial Z^{\le \frac{N}{2}}
  u_k\|_{L^\infty L^\infty(C^R_\tau)}
  \lesssim \|\la
  r\ra^{-\frac{1}{2}} \partial Z^{\le \frac{N}{2}+3}
    u_k\|_{L^2L^2(\tC^R_\tau)}
    \\
 +\|\la r\ra^{-\frac{\delta}{2}} Z^{\le N}\partial
  v_{k-1}\|_{L^\infty L^2} \|\la r\ra^{\frac{p-1}{2}} Z^{\le
    N} \good u_{k-1}\|_{L^2L^2}
\\+ \|Z^{\le N} \partial
u_{k-1}\|_{L^\infty L^2} \|\la r\ra^{\frac{p-1}{2}} Z^{\le N}\good
v_{k-1}\|_{L^2L^2}
\end{multline*}
 for $1\le R\le \tau/2$.
When this is combined with \eqref{commutators} and \eqref{Vk} it yields that
\begin{equation}
  \label{VIIIk}
  \sup_\tau \sup_{R\le \tau/2}
    \Bigl(\tau^{\frac{1}{2}} R \|\partial Z^{\le \frac{N}{2}}
    u_k\|_{L^\infty 
      L^\infty(C^R_\tau)}\Bigr) \lesssim
    \varepsilon + M_{k-1}^2.
  \end{equation}
  
Similarly using \eqref{weightedSob} and \eqref{commutators} in \eqref{product_rule_v} instead
gives
 \begin{multline*}
\|Z^{\le N} \Box v_k\|_{L^2L^2(\tC^R_\tau)} \lesssim
  R^{-\frac{1-\delta}{2}}\|\partial
  Z^{\le \frac{N}{2}+2} u_{k-1}\|_{L^\infty L^2} \|\la
  r\ra^{-\frac{1+\delta}{2}} \partial Z^{\le N}
  v_{k-1}\|_{L^2L^2(\tC^R_\tau)}
 \\ + R^{-\frac{1-\delta}{2}} \|\la r\ra^{-\frac{1+\delta}{2}} \partial
  Z^{\le \frac{N}{2}+2} v_{k-1}\|_{L^2L^2(\ttC^R_\tau)} \|\partial
  Z^{\le N} u_{k-1}\|_{L^\infty L^2}.
\end{multline*}
When combined with \eqref{CRT}, this yields
\begin{multline*}
  \tau^{\frac{1}{2}}R^{1-\frac{\delta}{2}} \|\partial Z^{\le \frac{N}{2}}
  v_k\|_{L^\infty L^\infty(C^R_\tau)}
  \lesssim \|\la r\ra^{-\frac{1+\delta}{2}}Z^{\le \frac{N}{2}+3}
    v_k\|_{L^2L^2(\tC^R_\tau)}
   \\ + \|\partial Z^{\le \frac{N}{2}+2} u_{k-1}\|_{L^\infty L^2} \|\la
  r\ra^{-\frac{1+\delta}{2}} \partial Z^{\le N}
  v_{k-1}\|_{L^2L^2(\tC^R_\tau)}
  \\+ \|\la r\ra^{-\frac{1+\delta}{2}} \partial
  Z^{\le \frac{N}{2}+2} v_{k-1}\|_{L^2L^2(\ttC^R_\tau)} \|\partial
  Z^{\le N} u_{k-1}\|_{L^\infty L^2},
\end{multline*}
which upon pairing with \eqref{VIk} gives
\begin{equation}
  \label{IXk}
  \Bigl[\sum_\tau \sum_{R\le \tau/4}
    \Bigl(R^{\frac{3-\delta}{2}} \|\partial Z^{\le \frac{N}{2}}
    v_k\|_{L^\infty 
      L^\infty(C^R_\tau)}\Bigr)^2\Bigr]^{\frac{1}{2}}
\lesssim \varepsilon + M_{k-1}^2.
\end{equation}

Using \eqref{CUT} in place of \eqref{CRT}, these same arguments show
\begin{multline*}
  U^{\frac{1}{2}}\tau \|\partial Z^{\le \frac{N}{2}}
  u_k\|_{L^\infty L^\infty(C^U_\tau)}
  \lesssim \|\la r\ra^{-\frac{1}{2}}
  Z^{\le \frac{N}{2}+3}
  u_k\|_{L^2L^2(\tC^U_\tau)}
 \\  +\|\la r\ra^{-\frac{\delta}{2}} Z^{\le N}\partial
  v_{k-1}\|_{L^\infty L^2} \|\la r\ra^{\frac{p-1}{2}} Z^{\le
    N} \good u_{k-1}\|_{L^2L^2}
\\+ \|Z^{\le N} \partial
u_{k-1}\|_{L^\infty L^2} \|\la r\ra^{\frac{p-1}{2}} Z^{\le N}\good
v_{k-1}\|_{L^2L^2},
\end{multline*}
and
\begin{multline*}
  U^{\frac{1}{2}}\tau^{1-\frac{\delta}{2}} \|\partial Z^{\le \frac{N}{2}}
  v_k\|_{L^\infty L^\infty(C^U_\tau)}
  \lesssim \|\la r\ra^{-\frac{1+\delta}{2}}
  Z^{\le \frac{N}{2}+3}
    v_k\|_{L^2L^2(\tC^U_\tau)}
  \\  + \|\partial Z^{\frac{N}{2}+2} u_{k-1}\|_{L^\infty L^2} \|\la
  r\ra^{-\frac{1+\delta}{2}} \partial Z^{\le N}
  v_{k-1}\|_{L^2L^2(\tC^U_\tau)}
 \\ + \|\la r\ra^{-\frac{1+\delta}{2}} \partial
  Z^{\le \frac{N}{2}+2} v_{k-1}\|_{L^2L^2(\ttC^U_\tau)} \|\partial
  Z^{\le N} u_{k-1}\|_{L^\infty L^2}.
\end{multline*}
When these are combined with \eqref{Vk} and \eqref{VIk} respectively,
we obtain
\begin{equation}
  \label{Xk}
 \sup_\tau \sup_{U\le \tau/4}
    \Bigl(\tau U^{\frac{1}{2}}\|\partial Z^{\le \frac{N}{2}}
    u_k\|_{L^\infty 
      L^\infty(C^U_\tau)}\Bigr)
  +\Bigl[\sum_\tau \sum_{U\le \tau/4}
    \Bigl(\tau^{1-\frac{\delta}{2}} U^{\frac{1}{2}}\|\partial Z^{\le \frac{N}{2}}
    v_k\|_{L^\infty 
      L^\infty(C^U_\tau)}\Bigr)^2\Bigr]^{\frac{1}{2}} \lesssim \varepsilon+M_{k-1}^2.
\end{equation}

The combination of \eqref{Ik}, \eqref{Vk}, \eqref{VIk}, \eqref{VIIIk},
\eqref{IXk}, and \eqref{Xk} prove \eqref{MkGoal} and, hence,
\eqref{MkBound} as desired.

  \subsection{Convergence}
It remains to show that the sequence $(u_k)$ and $(v_k)$ converge.  We
set  
\begin{multline}
  \label{A}
  A_k =
  \|\la r\ra^{\frac{p-1}{2}} \good Z^{\le N}(u_{k}-u_{k-1})\|_{L^2 L^2} +
  \|\la r\ra^{\frac{p-1}{2}} r^{-1} Z^{\le N} (u_{k}-u_{k-1})\|_{L^2L^2}
+
  \|\la r\ra^{\frac{p-1}{2}} \good Z^{\le N} (v_{k}-v_{k-1})\|_{L^2L^2} \\+
   \|\la
  r\ra^{\frac{p-1}{2}} r^{-1} Z^{\le N}(v_{k}-v_{k-1})\|_{L^2L^2}
  + \|Z^{\le N}(u_{k}-u_{k-1})\|_{LE^1} + \|\partial Z^{\le
    N}(u_{k}-u_{k-1})\|_{L^\infty L^2}
  \\+ \|\la r\ra^{-\frac{1+\delta}{2}} \partial Z^{\le N}
  (v_{k}-v_{k-1})\|_{L^2 L^2}
  +\sup_\tau \sup_{R\le \tau/2}
    \Bigl(\tau^{\frac{1}{2}} R \|\partial Z^{\le \frac{N}{2}}
    (u_k-u_{k-1})\|_{L^\infty 
      L^\infty(C^R_\tau)}\Bigr)
  \\+\Bigl[\sum_\tau \sum_{R\le \tau/2}
    \Bigl(R^{\frac{3-\delta}{2}} \|\partial Z^{\le \frac{N}{2}}
    (v_k-v_{k-1})\|_{L^\infty 
      L^\infty(C^R_\tau)}\Bigr)^2\Bigr]^{\frac{1}{2}}
 +\sup_\tau \sup_{U\le \tau/4}
    \Bigl(\tau U^{\frac{1}{2}}\|\partial Z^{\le \frac{N}{2}}
    (u_k-u_{k-1})\|_{L^\infty 
      L^\infty(C^U_\tau)}\Bigr)
\\+\Bigl[\sum_\tau \sum_{U\le \tau/4}
    \Bigl(\tau^{1-\frac{\delta}{2}} U^{\frac{1}{2}}\|\partial Z^{\le \frac{N}{2}}
    (v_k-v_{k-1})\|_{L^\infty 
      L^\infty(C^U_\tau)}\Bigr)^2\Bigr]^{\frac{1}{2}}.
\end{multline}
We seek to show that
\begin{equation}
  \label{AkGoal}
  A_k \le \frac{1}{2}A_{k-1},
\end{equation}
which implies that the sequences are Cauchy and thus convergent.

We note that
\begin{multline}
  \label{product_rule_u_difference}
  |Z^{\le N} \Box (u_k-u_{k-1})|\lesssim |\partial Z^{\le \frac{N}{2}}
  (u_{k-1}-u_{k-2})| |\partial Z^{\le N} v_{k-1}| + |\partial Z^{\le
    N} (u_{k-1}-u_{k-2})||\partial Z^{\le \frac{N}{2}}v_{k-1}| 
\\+ |\partial Z^{\le \frac{N}{2}} u_{k-2}| |\partial Z^{\le N}
(v_{k-1}-v_{k-2})|
+ |\partial Z^{\le N} u_{k-2}| |\partial Z^{\le \frac{N}{2}}(v_{k-1}-v_{k-2})|,
\end{multline}
\begin{multline}
  \label{product_rule_v_difference}
  |Z^{\le N} \Box (v_k-v_{k-1})|\lesssim |\partial Z^{\le \frac{N}{2}}
  (u_{k-1}-u_{k-2})| |\partial Z^{\le N} v_{k-1}| + |\partial Z^{\le
    N} (u_{k-1}-u_{k-2})||\partial Z^{\le \frac{N}{2}}v_{k-1}| 
\\+ |\partial Z^{\le \frac{N}{2}} u_{k-2}| |\partial Z^{\le N}
(v_{k-1}-v_{k-2})|
+ |\partial Z^{\le N} u_{k-2}| |\partial Z^{\le \frac{N}{2}}(v_{k-1}-v_{k-2})|,
\end{multline}
and
\begin{multline}
  \label{product_rule_u_null_difference}
  |\Box Z^{\le N} (u_k-u_{k-1})| \lesssim |Z^{\le \frac{N}{2}} \partial v_{k-1}|
  |Z^{\le N}\good (u_{k-1}-u_{k-2})| + |Z^{\le N} \partial v_{k-1}|
  |Z^{\le \frac{N}{2}}\good (u_{k-1}-u_{k-2})|
\\+ |Z^{\le \frac{N}{2}}\good u_{k-2}| |Z^{\le N} \partial
(v_{k-1}-v_{k-2})|
+ |Z^{\le N} \good u_{k-2}| |Z^{\le
  \frac{N}{2}}\partial(v_{k-1}-v_{k-2})|
\\+|Z^{\le \frac{N}{2}} \good v_{k-1}|
  |Z^{\le N}\partial (u_{k-1}-u_{k-2})| + |Z^{\le N} \good v_{k-1}|
  |Z^{\le \frac{N}{2}}\partial (u_{k-1}-u_{k-2})|
\\+|Z^{\le \frac{N}{2}} \good (v_{k-1}-v_{k-2})|
  |Z^{\le N}\partial u_{k-2}| + |Z^{\le N} \good (v_{k-1}-v_{k-2})|
  |Z^{\le \frac{N}{2}}\partial u_{k-2}|,
\end{multline}
which will be used in place of \eqref{product_rule_u},
\eqref{product_rule_v}, and \eqref{product_rule_u_null}
respectively.  Arguing as in the proof of \eqref{MkGoal} then shows
that
\[A_k \lesssim (M_{k-1}+M_{k-2}) A_{k-1}.\]
Provided that $\varepsilon$ is sufficiently small, an application of
\eqref{MkBound} immediately yields \eqref{AkGoal} and completes the proof.
 
We end with a brief remark about the asymptotics of the solution.  The
solution $u$ is also bounded in the norms given by \eqref{M}.  Indeed,
by examining the last two terms, one can immediately observe that $u$
as more rapid asymptotic decay $\O(t^{-1})$ than the component $v$,
which instead is $\O(t^{-1+\frac{\delta}{2}})$.

\bibliography{exterior}

\begin{thebibliography}{10}

\bibitem{Alinhac_ghostweight}
S.~Alinhac.
\newblock The null condition for quasilinear wave equations in two space
  dimensions {I}.
\newblock {\em Invent. Math.}, 145(3):597--618, 2001.

\bibitem{Christodoulou}
Demetrios Christodoulou.
\newblock Global solutions of nonlinear hyperbolic equations for small initial
  data.
\newblock {\em Comm. Pure Appl. Math.}, 39(2):267--282, 1986.

\bibitem{dafermos_rodnianski}
Mihalis Dafermos and Igor Rodnianski.
\newblock A new physical-space approach to decay for the wave equation with
  applications to black hole spacetimes.
\newblock In {\em X{VI}th {I}nternational {C}ongress on {M}athematical
  {P}hysics}, pages 421--432. World Sci. Publ., Hackensack, NJ, 2010.

\bibitem{Facci-Mcentarrfer-M}
Michael Facci, Alex Mcentarrfer, and Jason Metcalfe.
\newblock A $r^p$-weighted local energy approach to global existence for null
  form semilinear wave equations.
\newblock preprint, 2021.

\bibitem{Hidano-Yokoyama}
Kunio Hidano and Kazuyoshi Yokoyama.
\newblock Global existence for a system of quasi-linear wave equations in 3{D}
  satisfying the weak null condition.
\newblock {\em Int. Math. Res. Not. IMRN}, (1):39--70, 2020.

\bibitem{Hidano-Zha}
Kunio Hidano and Dongbing Zha.
\newblock Remarks on a system of quasi-linear wave equations in 3{D} satisfying
  the weak null condition.
\newblock {\em Commun. Pure Appl. Anal.}, 18(4):1735--1767, 2019.

\bibitem{John-Klainerman}
F.~John and S.~Klainerman.
\newblock Almost global existence to nonlinear wave equations in three space
  dimensions.
\newblock {\em Comm. Pure Appl. Math.}, 37(4):443--455, 1984.

\bibitem{John_counterexample}
Fritz John.
\newblock Blow-up for quasilinear wave equations in three space dimensions.
\newblock {\em Comm. Pure Appl. Math.}, 34(1):29--51, 1981.

\bibitem{KSS}
Markus Keel, Hart~F. Smith, and Christopher~D. Sogge.
\newblock Almost global existence for some semilinear wave equations.
\newblock {\em J. Anal. Math.}, 87:265--279, 2002.
\newblock Dedicated to the memory of Thomas H. Wolff.

\bibitem{Keir}
Joseph Keir.
\newblock The weak null condition and global existence using the p-weighted
  energy method.
\newblock {\em arXiv preprint arXiv:1808.09982}, 2018.

\bibitem{Klainerman}
S.~Klainerman.
\newblock The null condition and global existence to nonlinear wave equations.
\newblock In {\em Nonlinear systems of partial differential equations in
  applied mathematics, {P}art 1 ({S}anta {F}e, {N}.{M}., 1984)}, volume~23 of
  {\em Lectures in Appl. Math.}, pages 293--326. Amer. Math. Soc., Providence,
  RI, 1986.

\bibitem{klainermanSob}
Sergiu Klainerman.
\newblock Uniform decay estimates and the {L}orentz invariance of the classical
  wave equation.
\newblock {\em Comm. Pure Appl. Math.}, 38(3):321--332, 1985.

\bibitem{KlSid}
Sergiu Klainerman and Thomas~C. Sideris.
\newblock On almost global existence for nonrelativistic wave equations in
  {$3$}{D}.
\newblock {\em Comm. Pure Appl. Math.}, 49(3):307--321, 1996.

\bibitem{Lindblad-Rodnianski-1}
Hans Lindblad and Igor Rodnianski.
\newblock The weak null condition for {E}instein's equations.
\newblock {\em C. R. Math. Acad. Sci. Paris}, 336(11):901--906, 2003.

\bibitem{Lindblad-Rodnianski}
Hans Lindblad and Igor Rodnianski.
\newblock Global existence for the {E}instein vacuum equations in wave
  coordinates.
\newblock {\em Comm. Math. Phys.}, 256(1):43--110, 2005.

\bibitem{Lindblad-Tohaneanu}
Hans Lindblad and Mihai Tohaneanu.
\newblock Global existence for quasilinear wave equations close to
  {S}chwarzschild.
\newblock {\em Comm. Partial Differential Equations}, 43(6):893--944, 2018.

\bibitem{M-Rhoads}
Jason Metcalfe and Taylor Rhoads.
\newblock Long-time existence for systems of quasilinear wave equations.
\newblock {\em arXiv preprint arXiv:2203.08599}, 2022.

\bibitem{MS}
Jason Metcalfe and Christopher~D. Sogge.
\newblock Long-time existence of quasilinear wave equations exterior to
  star-shaped obstacles via energy methods.
\newblock {\em SIAM J. Math. Anal.}, 38(1):188--209, 2006.

\bibitem{ms_mathz}
Jason Metcalfe and Christopher~D. Sogge.
\newblock Global existence of null-form wave equations in exterior domains.
\newblock {\em Math. Z.}, 256(3):521--549, 2007.

\bibitem{MTT}
Jason Metcalfe, Daniel Tataru, and Mihai Tohaneanu.
\newblock Price's law on nonstationary space-times.
\newblock {\em Adv. Math.}, 230(3):995--1028, 2012.

\bibitem{Sideris_counterexample}
Thomas~C. Sideris.
\newblock Global behavior of solutions to nonlinear wave equations in three
  dimensions.
\newblock {\em Comm. Partial Differential Equations}, 8(12):1291--1323, 1983.

\bibitem{Sterbenz}
Jacob Sterbenz.
\newblock Angular regularity and {S}trichartz estimates for the wave equation.
\newblock {\em Int. Math. Res. Not.}, (4):187--231, 2005.
\newblock With an appendix by Igor Rodnianski.

\end{thebibliography}

\end{document}